\documentclass{article}

\usepackage{amssymb}	% Allows use of AMS's mathematical symbols
\usepackage{amsmath}    % Allows use of AMS's mathematical commands
\usepackage{setspace} % use \doublespacing after \maketitle
\usepackage{latexsym}	% Allows use of old latex symbols
\usepackage{longtable}
\usepackage{multirow}
\usepackage{epsfig}
\allowdisplaybreaks
 
\newtheorem{definition}{Definition}[section]
\newtheorem{algorithm}{Algorithm}[section]

\newtheorem{theorem}{Theorem}[section]

\newtheorem{proposition}{Proposition}[section]
\newtheorem{corollary}{Corollary}[section]
\newtheorem{remark}{Remark}[section]

\def\A{{\bf A}}
\def\B{{\bf B}}

\def\D{{\bf D}}
\def\E{{\bf E}}
\def\F{{\bf F}}
\def\G{{\bf G}}
\def\H{{\bf H}}
\def\I{{\bf I}}
\def\J{{\bf J}}

\def\L{{\bf L}}
\def\M{{\bf M}}

\def\0{{\bf 0}}
\def\P{{\bf P}}
\def\Q{{\bf Q}}
\def\R{{\bf R}}
\def\S{{\bf S}}
\def\T{{\bf T}}
\def\U{{\bf U}}
\def\V{{\bf V}}
\def\W{{\bf W}}

\def\Z{{\bf Z}}

\def\b{{\bf b}}

\def\e{{\bf e}}

\def\m{{\bf m}}

\def\q{{\bf q}}

\def\s{{\bf s}}
\def\t{{\bf t}}
\def\u{{\bf u}}

\def\x{{\bf x}}
\def\y{{\bf y}}
\def\z{{\bf z}}

\def\Tr{{\rm T}}

\def\diag{{\rm diag}}

\newenvironment{proof}[1][Proof:]{\begin{trivlist} 
\item[\hskip \labelsep {\bfseries #1}]}{\end{trivlist}} 
\newcommand{\qed}{\nobreak \ifvmode \relax \else \ifdim\lastskip<1.5em \hskip-\lastskip \hskip1.5em plus0em minus0.5em \fi \nobreak \vrule height0.75em width0.5em depth0.25em\fi}

\title{An Efficient Algorithm for Periodic Riccati Equation for Spacecraft Attitude Control Using Magnetic Torques}

\author{Yaguang Yang\thanks{Instrumentation and Control Engineer, Office of Research, US NRC, Two White Flint North
11545 Rockville Pike, Rockville, MD 20852-2738. Phone: (301) 415-0655. Email: yaguang.yang@verizon.net.} }
\begin{document}

\maketitle

\begin{abstract}
Spacecraft attitude control using only magnetic torques is a
periodic time-varying system as the Earth magnetic field in
the spacecraft body frame changes periodically while the 
spacecraft circles around
the Earth. The optimal controller design therefore
involves the solutions of the periodic Riccati 
differential or algebraic equations. This paper proposes an
efficient algorithm for the periodic discrete-time 
Riccati equation arising from a linear periodic time-varying 
system $(\A,\B)$, which explores and 
utilizes the fact that $\A$ is time-invariant and only $\B$ is time-varying in the system, a special 
properties associated with the problem of spacecraft attitude 
control using only magnetic torques.
\end{abstract}

\section{Introduction}
SPACECRAFT attitude control using only magnetic torques has several attractive features, such as low cost, high reliability 
(without moving mechanical parts), and seamless implementation. Therefore, numerous research papers were focused on
the problem of spacecraft attitude control using only magnetic 
torques in the last twenty five years (see 
\cite{me89,pitte93,wisni97,psiaki01,la04,sl05,la06,yra07,plv10,cw10,rh11,zl11} and references therein).
Because the Earth's magnetic field in the spacecraft body frame is 
approximately a periodic function as the spacecraft circles
around the Earth, the controller
design should be based on a time-varying system. Therefore, state space model is a natural choice.

Some researchers \cite{la04,sl05,la06} proposed direct design 
methods using Lyapunov stabilization theory. These designs use the nonlinear periodic model.
The existence of the solutions for these designs implicitly 
depends on the controllability for the nonlinear time-varying 
system. Therefore, Bhat \cite{bhat05} investigated controllability 
of the nonlinear time-varying systems. However, the condition for
the controllability of the nonlinear time-varying systems
established in \cite{bhat05} is hard to be verified and is a sufficient condition.
In addition, there is no systematic method for the selection of  Lyapunov functions.
Moreover, these designs do no consider the closed-loop system performances
other than the stability.

A more realistic design strategy is to use linearized time-varying 
system models. The standard design methods for these models
such as linear quadratic regulator (LQR) 
\cite{me89,pitte93,wisni97,psiaki01,cw10,plv10,rh11} and 
$\H_{\infty}$ control \cite{zl11} are discussed. 
But two important issues were not addressed in these papers. 
First, the existence of the solutions of LQR and $\H_{\infty}$
control directly depends on the controllability (or a slightly weak condition 
named stabilizability) of the linear time-varying system
which was not established for the spacecraft attitude control 
system using only magnetic torques. Second, the features and 
the structure of the spacecraft attitude control using magnetic 
torque were not explored. Instead, algorithms designed 
for general linear time-varying systems were used for this very specific problem. 
Therefore, those algorithms are not optimized for this 
problem.

The first issue is recently addressed in \cite{yang15} in which 
the conditions for the controllability of spacecraft attitude 
control using only magnetic torques were established. The second 
issue is the focus of this paper, we will explore the features 
and structure of the problem and propose an efficient algorithm
for the design of spacecraft attitude control using only magnetic torques.

There are two popular types of models used in spacecraft control system designs.
Some of the designs adopted Euler angle models \cite{me89,pitte93,wisni97,psiaki01,cw10} 
but others used the reduced quaternion models \cite{plv10,rh11,zl11}.
We will adopt a reduced quaternion model because of the merits of
the reduced quaternion model discussed in \cite{yang10,yang12,yang14}.

The remainder of the paper is organized as follows. Section 2
provides a description of the linear time-varying model of the 
spacecraft attitude control system using only magnetic torque.
Section 3 derives the controller design algorithm for the linear 
time-varying system. Section 4 presents a simulation example to 
demonstrate the effectiveness and efficiency of the design 
algorithm. The conclusions are summarized in Section 5.

\section{Spacecraft Model}
The linearized continuous-time model for spacecraft attitude control using only magnetic torques
can be expressed in a reduced quaternion form. 
Let $\J$ be the inertia matrix of a spacecraft defined by
\begin{eqnarray}
\J =\left[   \begin{array}{ccc}
J_{11} & J_{12} & J_{13} \\
J_{21} & J_{22} & J_{23} \\
J_{31} & J_{32} & J_{33}   
\end{array} \right].
\label{inertia}
\end{eqnarray}
We will consider the nadir pointing spacecraft. Therefore, the attitude of the spacecraft is represented by the 
rotation of the spacecraft body frame relative to the local vertical and local horizontal 
(LVLH) frame.  Let 
${\omega}=[\omega_1,\omega_2, \omega_3]^{\Tr}$ be the 
body rate with respect to the LVLH frame represented in the body frame, 
$\omega_0$ be the orbit (and LVLH frame) rate with respect to 
the inertial frame, represented in the LVLH frame.
Let $\bar{\q}=[q_0, q_1, q_2, q_3]^{\Tr}=[q_0, \q^{\Tr}]^{\Tr}=
[\cos(\frac{\alpha}{2}), \hat{\e}^{\Tr}\sin(\frac{\alpha}{2})]^{\Tr}$
be the quaternion representing the rotation of the body frame relative to the LVLH frame, 
where $\hat{\e}$ is the unit length rotational axis and $\alpha$ is the rotation angle about $\hat{\e}$.  
The control torques generated by magnetic coils interacting with
the Earth's magnetic field is given by (see \cite{sidi97})
\[
\u=\m \times \b
\]
where the Earth's magnetic field in spacecraft coordinates, 
$\b(t)=[b_1(t),b_2(t),b_3(t)]^{\Tr}$, is computed using the 
spacecraft position, the spacecraft attitude, and a spherical harmonic model of the Earth's magnetic field \cite{wertz78}; and 
$\m=[m_1,m_2,m_3]^{\Tr}$ is the spacecraft magnetic coils' induced 
magnetic moment in the spacecraft coordinates.
The time-variation of the system is an approximate periodic function of $\b(t)=\b(t+P)$ where the orbital period
is given by \cite{sidi97}
\begin{equation}
P=\frac{2\pi}{\omega_0}=2\pi \sqrt{\frac{a^3}{GM}},
\label{period}
\end{equation}
where $a$ is the orbital radius (for circular orbit) and $GM=3.986005*10^{14}{m^3/s^2}$ \cite{wertz78}.
 This magnetic field $\b(t)$ can be 
approximately expressed as follows \cite{psiaki01}:
\begin{equation}
\left[  \begin{array}{c} 
b_1(t) \\
b_2(t) \\
b_3(t) 
\end{array}  \right] 
= \frac{\mu_f}{a^{3}}
\left[  \begin{array}{c} 
\cos(\omega_0t)\sin(i_m) \\
-\cos(i_m) \\
2\sin(\omega_0t)\sin(i_m) 
\end{array}  \right],
\label{field}
\end{equation}
where $i_m$ is the inclination of the spacecraft orbit with 
respect to the magnetic equator, $\mu_f=7.9 \times 10^{15}$
Wb-m is the field's dipole strength. The time $t=0$ is measured at the ascending-node
crossing of the magnetic equator. Then, the reduced
quaternion linear time-varying system is given as follows \cite{yang15}:
\begin{eqnarray}
\left[  \begin{array}{c} 
\dot{q}_1 \\
\dot{q}_2 \\
\dot{q}_3 \\
\dot{\omega}_1 \\
\dot{\omega}_2 \\ 
\dot{\omega}_3 
\end{array}  \right] 
&=& \left[  \begin{array}{cccccc} 
0 & 0 & 0 & .5 & 0 & 0 \\
0 & 0 & 0 & 0 & .5 & 0 \\
0 & 0 & 0 & 0 & 0 & .5 \\
f_{41} & 0 & 0 & 0 & 0 & f_{46} \\
0 & f_{52} & 0 & 0 & 0 & 0 \\ 
0 & 0 & f_{63} & f_{64} & 0 & 0 
\end{array}  \right] 
\left[  \begin{array}{c} 
{q}_1 \\
{q}_2 \\
{q}_3 \\
{\omega}_1 \\
{\omega}_2 \\ 
{\omega}_3 
\end{array}  \right]
+\left[  \begin{array}{ccc} 
0 & 0 & 0 \\
0 & 0 & 0  \\
0 & 0 & 0  \\
0 & \frac{b_3(t)}{J_{11}} & -\frac{b_2(t)}{J_{11}} \\
-\frac{b_3(t)}{J_{22}} & 0 & \frac{b_1(t)}{J_{22}} \\ 
\frac{b_2(t)}{J_{33}} & -\frac{b_1(t)}{J_{33}} & 0
\end{array}  \right] 
\left[  \begin{array}{c} 
{m}_1 \\
{m}_2 \\
{m}_3 
\end{array}  \right] \nonumber \\
& := & 
\left[ \begin{array}{cc}
\0_3 & \frac{1}{2} \I_3  \\ \boldsymbol{\Lambda}_1 & \boldsymbol{\Sigma}_1
\end{array} \right]
\left[ \begin{array}{c}
\q \\ \boldsymbol{\omega} 
\end{array} \right]
+ \left[ \begin{array}{c}
\0_3 \\ \B_2(t)
\end{array} \right] \m
=\A\x+\B(t) \m,
\label{varying}
\end{eqnarray}
where
\begin{eqnarray}
\A= \left[ \begin{array}{cc}
\0_3 & \frac{1}{2} \I_3  \\ \boldsymbol{\Lambda}_1 & \boldsymbol{\Sigma}_1
\end{array} \right],
\hspace{0.1in}
&
\B=\left[ \begin{array}{c}
\0_3 \\ \B_2(t)
\end{array} \right],
\label{AB}
\end{eqnarray}
\begin{equation}
\B_2(t)= \left[ \begin{array}{ccc}
0 & b_{42}(t) & b_{43}(t) \\
b_{51}(t) & 0 & b_{53}(t) \\ 
b_{61}(t) & b_{62}(t)  & 0
\end{array} \right],
\label{bt}
\end{equation}
\begin{eqnarray}
f_{41}=[8(J_{33}-J_{22})\omega_0^2]/J_{11} \\
f_{46}=(-J_{11}+J_{22}-J_{33})\omega_0/J_{11} \\
f_{64}=(J_{11}-J_{22}+J_{33})\omega_0/J_{33} \\
f_{52}=[6(J_{33}-J_{11})\omega_0^2]/J_{22} \\
f_{63}=[2(J_{11}-J_{22})\omega_0^2]/J_{33}
\label{para}
\end{eqnarray}
and 
\begin{eqnarray}
b_{42} (t) = \frac{2\mu_f}{a^3 J_{11}} \sin(i_m) \sin(\omega_0t) 
\\
b_{43} (t)  = \frac{\mu_f}{a^3 J_{11}} \cos(i_m) 
\\
b_{53} (t)  = \frac{\mu_f}{a^3 J_{22}} \sin(i_m) \cos(\omega_0t) 
\\
b_{51} (t)  = -\frac{2\mu_f}{a^3 J_{22}} \sin(i_m) \sin(\omega_0t) 
= -b_{42}\frac{J_{11}}{J_{22}} 
\\
b_{61} (t) =-\frac{\mu_f}{a^3 J_{33}} \cos(i_m) = -b_{43} \frac{J_{11}}{J_{33}}
\\
b_{62} (t) =-\frac{\mu_f}{a^3 J_{33}} \sin(i_m) \cos(\omega_0t)
=-b_{53}\frac{J_{22}}{J_{33}}.
\label{bij}
\end{eqnarray}
It is easy to verify that $\det{\A}=(\frac{1}{2})^3\det(\boldsymbol{\Lambda}_1)$ and $\A$ is nonsingular if $J_{11} \neq J_{22}$, $J_{11} \neq J_{33}$, and $J_{33} \neq J_{22}$.

Oftentimes, a spacecraft controller is implemented in a computer 
system which is a discrete system. Therefore, the following 
discrete model is used for the design in practical implementation:
\begin{eqnarray}
\x_{k+1}= \A_k \x_k +\B_k \m_k.
\label{discrete}
\end{eqnarray}
The system matrices $(\A_k, \B_k)$ in the discrete model can be derived from 
(\ref{varying}), (\ref{AB}), and (\ref{bt}) by different methods. Let $t_s$ be the sample time, we use the 
following formulations.
\begin{eqnarray}
\A_k= (\I+\A t_s), \hspace{0.1in}
\B_k = \B(k t_s) t_s.
\label{ABk}
\end{eqnarray}
Note that 
\[
\det(\I+\A t_s)=\det \left[ \begin{array}{ccc}
\I & & 0.5t_s \I \\ t_s \boldsymbol{\Lambda}_1 & & \I + t_s\boldsymbol{\Sigma}_1
\end{array} \right]
=\det \left[ \begin{array}{ccc}
\I & & 0.5t_s \I \\ \0_3 & & \I + t_s\boldsymbol{\Sigma}_1-0.5t_s^2 \boldsymbol{\Lambda}_1
\end{array} \right]
\]
which is invertiable as long as $t_s$ is selected small enough.
It is worthwhile to mention that in  both continuous-time and discrete-time models, the time-varying feature is
introduced by time-varying matrices $\B(t)$ or $\B_k$; the system matrices
$\A$ and $\A_k$ are constants and invertiable which are important for us to derive an efficient computational algorithm.

\section{Computational Algorithm for the LQR Design}
It is well-known that the LQR design relies on the solution of 
either the 
differential Riccati equation (for continuous-time systems) or 
the discrete Riccati equation (for discrete-time systems) 
\cite{lvs12}. If a system is periodic, such as (\ref{varying}) or
(\ref{discrete}), the LQR design relies on the periodic 
solution of the periodic Riccati equation \cite{bittanti91}. 
Our discussion about the computational algorithm is focused
on the solution of periodic discrete Riccati equation 
using the special properties of (\ref{discrete}), i.e., $\A_k$
is constant and invertiable for $\forall k$.

\subsection{Preliminary Results}
Let 
\begin{equation}
\L = \left[ \begin{array}{rclc}
\0  & & \I & \\
-\I & & \0 &
\end{array} \right] \in \R^{2n \times 2n},
\label{L}
\end{equation}
where $n$ is the dimension of $\A$ or $\A_k$ in general case and $n=6$ in (\ref{varying}) and
(\ref{discrete}). Note that $\L^{\Tr}=\L^{-1}=-\L$. Two types of matrices defined below are important to 
the solutions of the Riccati equations.
\begin{definition}[{\cite{laub79}}]
A matrix $\M  \in \R^{2n \times 2n}$ is Hamiltonian if $\L^{-1}\M^{\Tr}\L= -\M$. 
A matrix $\M  \in \R^{2n \times 2n}$ is symplectic if $\L^{-1}\M^{\Tr}\L= \M^{-1}$. 
\end{definition}
\begin{proposition}
If $\M_1$ and $\M_2$ are symplectic, then $\M_1 \M_2$ is symplectic.
\label{prop1}
\end{proposition}
\begin{proof}
Since $\L^{-1}\M_1^{\Tr}\L=\M_1^{-1}$ and $\L^{-1}\M_2^{\Tr}\L=\M_2^{-1}$,
we have 
\[
\begin{array}{cc}
\L^{-1}(\M_1\M_2)^{\Tr}\L = \L^{-1}\M_2^{\Tr}\M_1^{\Tr}\L
=\L^{-1}\M_2^{\Tr}\L \L^{-1} \M_1^{\Tr}\L =\M_2^{-1}\M_1^{-1}=(\M_1\M_2)^{-1}.
\end{array}
\]
This concludes the proof.
\hfill \qed
\end{proof}
In the sequel, we use $\lambda(\M)$ or simply $\lambda$ for an eigenvalue of a matrix $\M$ and $\sigma(\M)$
for the set of all eigenvalues of $\M$. The following two theorems play essential roles.
\begin{theorem}[{\cite{vaughan70,laub72}}]
Let $\M  \in \R^{2n \times 2n}$ is Hamiltonian. Then $\lambda \in \sigma(\M)$ implies
$-\lambda \in \sigma(\M)$ with the same multiplicity. 
Let $\M  \in \R^{2n \times 2n}$ is symplectic. Then $\lambda \in \sigma(\M)$  implies
$\frac{1}{\lambda} \in \sigma(\M)$ with the same multiplicity. 
\label{thm1}
\end{theorem}
\begin{theorem}[{\cite{mw31}}]
Let $\M\in \R^{n \times n}$.  Then there exists an orthogonal similarity transformation
$\U$ such that $\U^{\Tr}\M \U$ is quasi-upper-triangular. Moreover, $\U$ can be chosen
such that the $2\times 2$ and $1\times 1$ diagonal blocks appear in any desired order.
\label{thm2}
\end{theorem}

Theorem \ref{thm2} is the so called real Schur decomposition.
Combining the above two theorems, we conclude that 
\begin{corollary}
Let $\M  \in \R^{2n \times 2n}$ is Hamiltonian or symplectic. Then there exists an orthogonal 
similarity transformation $\U$ such that
\begin{equation}
\left[ \begin{array}{cc} \U_{11} &  \U_{12}  \\
\U_{21}   &    \U_{22}
\end{array} \right]^{\Tr}\M 
\left[ \begin{array}{cc} \U_{11} &  \U_{12}  \\
\U_{21}   &    \U_{22}
\end{array} \right]=
\left[ \begin{array}{cc}
\S_{11} &  \S_{12}  \\
\0   &    \S_{22}
\end{array} \right]
\label{decomp}
\end{equation}
where $\U_{11},\U_{12},\U_{21},\U_{22},\S_{11},\S_{12},\S_{22} \in 
\R^{n \times n}$, and $\S_{11}$, $\S_{22}$ are quasi-upper-triangular. Moreover, 
\begin{itemize}
\item[1] if $\M$ is Hamiltonian, then $\sigma(\S_{11}) \le 0$
(or $\sigma(\S_{22}) \ge 0$) and $\sigma(\S_{22}) \ge 0$ (or $\sigma(\S_{11}) \le 0$).
\item[2] if $\M$ is symplectic, then $\sigma(\S_{11})$ lies
inside (or outside) the unit circle and $\sigma(\S_{22})$ lies outside (or inside) the unit circle.
\end{itemize}
\label{cor1}
\end{corollary}
The Hamiltonian matrix is used in the derivation of the solution for the continuous-time differential
Riccati equation, while the symplectic matrix is used in the 
derivation of the solution for the discrete-time
algebraic Riccati equation. In our discussion, therefore, 
the symplectic matrix plays a fundamental role.

\subsection{Solution of the Riccati Algebraic Equation}
For a discrete linear time-varying system (\ref{discrete}), the LQR state feedback control is 
to find the optimal $\m_k$ to minimize the following quadratic cost function
\begin{equation}
\lim_{N \rightarrow \infty} \left( \min  \frac{1}{2} \x_N^{\Tr} \Q_N  \x_N
+\frac{1}{2} \sum_{k=0}^{N-1} \x_k^{\Tr} \Q_k  \x_k+ \m_k^{\Tr} \R_k  \m_k
\right)
\label{contCost}
\end{equation}
where 
\begin{eqnarray}
\Q_k\ge 0, \\
\R_k> 0,
\end{eqnarray}
and the initial condition $\x_0$ is given.
The controllability of spacecraft attitude control using only
magnetic torques is discussed in \cite{yang15}. Therefore, if $(\A_k, \Q_k)$ is 
detectable or $\Q_k >0$, the optimal feedback $\m_k$  is given by 
\cite{lvs12,hl94}
\begin{equation}
\m_k = -(\R_k+\B_k^{\Tr}\P_{k+1}\B_k)^{-1} \B^{\Tr}_k  \P_{k+1} \A_k \x_k,
\label{optiSolu}
\end{equation}
where $\P_{k}$ is the unique positive semi-definite solution of the 
following discrete Riccati equation \cite{lvs12,laub79,hl94}
\begin{equation}
\P_k =  \Q_k+\A^{\Tr}_k\P_{k+1}\A_k -\A_k^{\Tr} \P_{k+1} \B_k(\R_k+\B_k^{\Tr}\P_{k+1}\B_k)^{-1}\B_k^{\Tr} \P_{k+1}\A_k ,
\label{Riccati}
\end{equation}
with the boundary condition $\P_N=\Q_N$. For this discrete Riccati equation (not necessarily periodic) 
given as (\ref{Riccati}),  
it can be solved using a symplectic system associated with (\ref{discrete}) and (\ref{contCost})
as follows \cite{laub79,hl94,pls80}.
\begin{equation}
\E_k \z_{k+1}=\E_k \left[ \begin{array}{c} \x_{k+1} \\ \y_{k+1} \end{array} \right]
=\F_k \left[ \begin{array}{c} \x_{k} \\ \y_{k} \end{array} \right]
=\F_k \z_k
\label{recursive}
\end{equation}
where $\x_k$ is the state and $\y_k$ is the costate, 
\begin{equation}
\E_k =  \left[ \begin{array}{cc} 
\I & \B_k \R_k^{-1} \B_k^{\Tr}  \\ \0  & \A_{k}^{\Tr}
\end{array} \right],
\label{Ek}
\end{equation}
\begin{equation}
\F_k =  \left[ \begin{array}{cc} 
\A_k & \0  \\ -\Q_k  & \I
\end{array} \right].
\label{Fk}
\end{equation}
If $\E_k$ is invertiable which is true for 
$\det(\I+ t_s \boldsymbol{\Sigma}_1 -\frac{1}{2} t_s^2 
\boldsymbol{\Lambda}_1) \neq 0$,
\[
\E_k^{-1} =\left[ \begin{array}{cc}
\I & - \B_k \R_k^{-1} \B_k^{\Tr} \A_k^{-\Tr} \\
\0 & \A^{-\Tr}_k 
\end{array} \right],
\]
a symplectic matrix can be formed \cite{pls80} as 
\begin{eqnarray}
\Z & = & \E^{-1}_k \F_k = \left[ \begin{array}{cc} 
\I & - \B_k \R_k^{-1} \B_k^{\Tr} \A_k^{-\Tr} \\
\0 & \A^{-\Tr}_k 
\end{array} \right]
 \left[ \begin{array}{cc} 
\A_k & \0  \\ -\Q_k  & \I
\end{array} \right] \nonumber \\
& = & 
 \left[ \begin{array}{cc} 
\A_k+\B_k\R_k^{-1}\B_k^{\Tr} \A_k^{-\Tr} \Q_k  & -\B_k\R_k^{-1}\B_k^{\Tr} \A_k^{-\Tr}  
\\ -\A_k^{-\Tr} \Q_k  &  \A_k^{-\Tr}
\end{array} \right].
\end{eqnarray}
It is straghtforward to verify that $\L^{-1}\Z^{\Tr} \L=\Z^{-1}$,
therefore, from Corollary \ref{cor1}, there exists an orthogonal
matrix $\U$ such that 
\begin{equation}
\left[ \begin{array}{cc} \U_{11} &  \U_{12}  \\
\U_{21}   &    \U_{22}
\end{array} \right]^{\Tr}\Z
\left[ \begin{array}{cc} \U_{11} &  \U_{12}  \\
\U_{21}   &    \U_{22}
\end{array} \right]=
\left[ \begin{array}{cc}
\S_{11} &  \S_{12}  \\
\0   &    \S_{22}
\end{array} \right].
\end{equation}
The (steady state) solution of (\ref{Riccati}) is given as 
follows \cite[Theorem 6]{laub79}.
\begin{theorem}
$\U_{11}$ is invertiale and $\P=\U_{12}\U_{11}^{-1}$ solves (\ref{Riccati}) with $\P=\P^{\Tr} \ge 0$;
\begin{eqnarray}
\sigma(\S_{11}) & = & \sigma(\A_l-\B_k(\R_k+\B_k^{\Tr} \P_k\B_k)^{-1} \B_k^{\Tr} \P_k\A_k )
\nonumber \\
& = & \sigma(\A_k-\B_k\R_k^{-1}\B_k^{\Tr} \A_k^{-\Tr}(\P_k-\Q_k) )
\nonumber \\
& = & \sigma(\A_k-\B_k\R_k^{-1}\B_k^{\Tr} (\P_k^{-1}-\B_k\R_k^{-1}\B_k^{\Tr} )^{-1} \A_k )
\nonumber \\
& = & \mbox{the ``closed-loop'' spectrum}.
\end{eqnarray}
\label{laub6}
\end{theorem}

\subsection{Solution of the Periodic Riccati Algebraic Equation}
Now, we consider the periodic time-varying system where
\begin{eqnarray}
\A_k=\A_{k+1}= \ldots =\A_{k+p}, \\
\B_k=\B_{k+p}, \\
\Q_k=\Q_{k+1}= \ldots =\Q_{k+p} \ge 0, \\
\R_k=\R_{k+p} > 0,
\end{eqnarray}
only $\B_k$ (and possibly $\R_k$) are priodic with period $p$. 
It is worthwhile to mention that 
$\A_k$ and $\Q_k$ are actually constant matrices. 
The optimal feedback given by
(\ref{optiSolu}) is periodic with $\P_{k}=\P_{k+p}$, a unique 
periodic positive seme-definite solution of the periodic Riccati 
equation (cf. \cite{bittanti91}). Therefore, 
using the similar process for general discrete Riccati equation,
and noticing that $\F_k=\F$ is a constant matrix because 
$\A_k$ and $\Q_k$ are constant matrices, we have 
\begin{eqnarray}
\E_k \z_{k+1} & = & \F \z_k \\
\E_{k+1} \z_{k+2} & = & \F \z_{k+1} \\
& \vdots &  \\
\E_{k+p-1} \z_{k+p} & = & \F \z_{k+p-1}.
\end{eqnarray}
This gives
\begin{equation}
\z_{k+p}= \boldsymbol{\Pi}_k \z_k,
\label{lump}
\end{equation}
with
\begin{equation}
\boldsymbol{\Pi}_k =\E_{k+p-1}^{-1} \F \ldots \E_{k+1}^{-1} \F\E_k^{-1} \F.
\label{pi}
\end{equation}
Using Proposition \ref{prop1}, we conclude that 
$\boldsymbol{\Pi}_k$ is a symplectic matrix. Therefore, 
there is an orthogonal matrix $\T_k$ such that
\begin{equation}
\left[ \begin{array}{cc} \T_{11k} &  \T_{12k}  \\
\T_{21k}   &    \T_{22k}
\end{array} \right]^{\Tr}\boldsymbol{\Pi}_k
\left[ \begin{array}{cc} \T_{11k} &  \T_{12k}  \\
\T_{21k}   &    \T_{22k}
\end{array} \right]=
\left[ \begin{array}{cc}
\S_{11k} &  \S_{12k}  \\
\0   &    \S_{22k}
\end{array} \right].
\label{schur}
\end{equation}
Finally, using Theorem \ref{laub6}, we have, for each sampling period $k \in \{ 0, 1, \ldots, p-1 \}$
the steady state solution of the Riccati equation corresponding to (\ref{lump}) is given by
\begin{equation}
\P_{k}= \T_{21k}\T_{11k}^{-1}.
\label{Xk}
\end{equation}
In view of that $\F$ is invertiable in the problem of spacecraft attitude
control using only magnetic torques, this method is more efficient than the one in \cite{pls80} because the latter is designed for singular 
$\F$. But the method of calculating (\ref{pi}), (\ref{schur}), and 
(\ref{Xk}) as described above (proposed in \cite{hl94}) is not 
the best way for the
problem of spacecraft attitude control using only magnetic 
torques. As a matter of the fact, equation (\ref{lump}) 
can be written as 
\begin{equation}
\left[ \begin{array}{c}
\x_k \\ \y_k
\end{array} \right]
=\z_k=\boldsymbol{\Gamma}_k \z_{k+p}
=\boldsymbol{\Gamma}_k \left[ \begin{array}{c}
\x_{k+p} \\ \y_{k+p} 
\end{array} \right]
\label{inverse}
\end{equation}
with the initial state $\x_0$, the boundary condition \cite{lvs12}
\begin{equation}
\y_N=\Q_N \x_N,
\label{boundary1}
\end{equation}
and
\begin{equation}
\boldsymbol{\Gamma}_k = \F^{-1} \E_{k} \F^{-1} \E_{k+1} \ldots, 
\F^{-1} \E_{k+p-2} \F^{-1} \E_{k+p-1}.
\label{Gamma}
\end{equation}
\begin{remark}
It is worthwhile to note that forming $\boldsymbol{\Gamma}_k$ needs no inversion of $\E_k$ for any $k$
and $\F^{-1}$ needs to be computed only one time. Therefore, the computation 
of $\boldsymbol{\Gamma}_k$ is much more efficient than the computation of 
$\boldsymbol{\Pi}_k$. We will show that the rest computation will be similar 
to the method proposed in \cite{laub79}).
\end{remark}
Since
\[
\F^{-1}=
\left[ \begin{array}{cc} 
\A_k^{-1}  & \0  \\ \Q_k  \A_k^{-1}   & \I
\end{array} \right],
\]
\begin{eqnarray}
\M & = & \F^{-1} \E_k =  
 \left[ \begin{array}{cc} 
\A_k^{-1}  & \0  \\ \Q_k  \A_k^{-1}   & \I
\end{array} \right] 
\left[ \begin{array}{cc} 
\I & \B_k \R_k^{-1} \B_k^{\Tr}  \\ \0  & \A_k^{\Tr}
\end{array} \right]
\nonumber \\
& = & 
 \left[ \begin{array}{ccc} 
\A_k^{-1} & & \A_k^{-1} \B_k\R_k^{-1}\B_k^{\Tr} \\
\Q_k \A_k^{-1} & & \Q_k \A_k^{-1} \B_k\R_k^{-1}\B_k^{\Tr} + \A_k^{\Tr} 
\end{array} \right],
\label{M}
\end{eqnarray}
which is a similar formula as given in \cite{vaughan70}. 
It is straghtforward to verify that $\M$ is symplectic.
In fact,
\begin{eqnarray}
\L^{-1}\M^{\Tr} \L & = & \left[ \begin{array}{cc} 
\0 & -\I  \\ \I  & \0
\end{array} \right]
 \left[ \begin{array}{cc} 
\A_k^{-\Tr} & \A_k^{-\Tr} \Q_k    \\
 \B_k\R_k^{-1}\B_k^{\Tr} \A_k^{-\Tr}  &  \A_k +\B_k\R_k^{-1}\B_k^{\Tr} \A_k^{-\Tr} \Q_k 
\end{array} \right] \L
\nonumber \\
& = & 
 \left[ \begin{array}{cc} 
- \B_k\R_k^{-1}\B_k^{\Tr} \A_k^{-\Tr}  
&  -\A_k   - \B_k\R_k^{-1}\B_k^{\Tr} \A_k^{-\Tr} \Q_k \\
\A_k^{-\Tr} & \A_k^{-\Tr} \Q_k    
\end{array} \right] 
 \left[ \begin{array}{cc} 
\0 & \I  \\ -\I  & \0
\end{array} \right]
\nonumber \\
& = & 
 \left[ \begin{array}{cc} 
\A_k  + \B_k\R_k^{-1}\B_k^{\Tr} \A_k^{-\Tr} \Q_k   
& -  \B_k\R_k^{-1}\B_k^{\Tr} \A_k^{-\Tr}   \\
-\A_k^{-\Tr} \Q_k     &     \A_k^{-\Tr} 
\end{array} \right] 
\nonumber \\
& = &  \M^{-1}.
\end{eqnarray}
Since $\M$ is symplectic, using Proposition \ref{prop1} 
again, $\boldsymbol{\Gamma}_k$ is symplectic. Let
\[ 
\V_k=\left[ \begin{array}{cc} 
\V_{11k} & \V_{12k} \\ \V_{21k} & \V_{22k}
\end{array} \right]
\] 
be a matrix that transform $\boldsymbol{\Gamma}_k$ into a Jordon form, we have 
\begin{equation}
\boldsymbol{\Gamma}_k \V_k =\V_k \left[ \begin{array}{cc} 
\boldsymbol{\Delta}_k & \0  \\ \0  & \boldsymbol{\Delta}_k^{-1} 
\end{array} \right]
\label{jordan}
\end{equation}
where $\boldsymbol{\Delta}_k$ is the Jordan block matrix of the $n$ eigenvalues
outside of the unit circle. One of the main results 
of this paper is the following theorem.
\begin{theorem}
The solution of the Riccati
equation corresponding to (\ref{inverse}) is given by
\begin{equation}
\P_k =\V_{21k}\V_{11k}^{-1}, \hspace{0.1in} k=0, \ldots, p-1.
\label{eigenSolu}
\end{equation}
\end{theorem}
\begin{proof}
The proof uses similar ideas of \cite{vaughan70,lvs12}.
Since the periodicity of the system, the Riccati equation 
corresponding to (\ref{inverse}) represents any one of
$k \in \{0,1,\ldots, p-1 \}$ equations which
has a sample period increasing by $p$ with the
patent $k, k+p, k+2p, \ldots, k+\ell p, \ldots$.
In the following discussion, 
we consider one Riccati equation and drop the
subscript $k$ to simplify the notation to
$0,  p,  2p, \ldots, \ell p, \ldots$. 
To make the notation simpler, we will drop $p$
and use $\ell$ for this step increment. Assume that the solution has the form
\begin{equation}
\y_{\ell} = \P \x_{\ell}.
\label{assumption}
\end{equation}
Further, we assume for simiplisity that the eigenvalues of 
$\boldsymbol{\Gamma}$ are distinct, therefore, 
$\boldsymbol{\Delta}$ is diagonal. For any integer $\ell \ge 0$, 
let
\begin{equation}
\left[ \begin{array}{c}
\x_{\ell} \\ \y_{\ell}  
\end{array} \right]
= \left[ \begin{array}{cc} 
\V_{11} & \V_{12} \\ \V_{21} & \V_{22}
\end{array} \right]
\left[ \begin{array}{c}
\t_{\ell}  \\ \s_{\ell} 
\end{array} \right],
\label{tmp1}
\end{equation}
from (\ref{inverse}), (\ref{jordan}) and (\ref{tmp1}), we have 
\[
\V \left[ \begin{array}{c}
\t_{\ell}  \\ \s_{\ell}  
\end{array} \right]
=\left[ \begin{array}{c}
\x_{\ell}  \\ \y_{\ell}  
\end{array} \right]
=\boldsymbol{\Gamma} \left[ \begin{array}{c}
\x_{\ell+1}  \\ \y_{\ell+1} 
\end{array} \right]
= \boldsymbol{\Gamma}  \V  \left[ \begin{array}{c}
\t_{\ell+1}  \\ \s_{\ell+1}  
\end{array} \right]
=\V \left[ \begin{array}{cc} 
\boldsymbol{\Delta}  & \0  \\ \0  & \boldsymbol{\Delta}^{-1} 
\end{array} \right]
\left[ \begin{array}{c}
\t_{\ell+1} \\ \s_{\ell+1} 
\end{array} \right],
\]
which is equivalent to
\[
\left[ \begin{array}{c}
\t_{\ell} \\ \s_{\ell} 
\end{array} \right]
= \left[ \begin{array}{cc} 
\boldsymbol{\Delta}  & \0  \\ \0  & \boldsymbol{\Delta}^{-1} 
\end{array} \right]
\left[ \begin{array}{c}
\t_{\ell+1} \\ \s_{\ell+1} 
\end{array} \right].
\]
Hence,
\begin{equation}
\left[ \begin{array}{c}
\t_{\ell} \\ \s_{\ell} 
\end{array} \right]
= \left[ \begin{array}{cc} 
\boldsymbol{\Delta}^{N-\ell}  & \0  \\ \0  & \boldsymbol{\Delta}^{-(N-\ell)} 
\end{array} \right]
\left[ \begin{array}{c}
\t_{N} \\ \s_{N} 
\end{array} \right],
\label{nstep}
\end{equation}
Using the boundary condition (\ref{boundary1}) and (\ref{tmp1}), 
we have
\[
\Q_N (\V_{11 } \t_{N} +  \V_{12} \s_{N} )
=\Q_N \x_N= \y_N=  \V_{21} \t_{N} +  \V_{22} \s_{N} ,
\]
this gives
\[
-(\V_{21} -\Q_N \V_{11}) \t_N =  (\V_{22} -\Q_N \V_{12}) \s_N,
\]
or equivalently
\begin{equation}
\s_N =-(\V_{22} -\Q_N \V_{12} )^{-1}(\V_{21} -\Q_N \V_{11}) \t_N:=\H \t_N.
\label{kmatrix}
\end{equation}
Combining (\ref{nstep}) and (\ref{kmatrix}) yields
\[
\s_{\ell} = \boldsymbol{\Delta}^{-(N-\ell)} \s_N =\boldsymbol{\Delta}^{-(N-\ell)} \H \t_N 
=\boldsymbol{\Delta}^{-(N-\ell)} \H \boldsymbol{\Delta}^{-(N-\ell)} \t_{\ell}:=\G \t_{\ell},
\]
with $\G =\boldsymbol{\Delta}^{-(N-\ell)} \H \boldsymbol{\Delta}^{-(N-\ell)}$. Finally, 
using this relation, (\ref{tmp1}), and 
(\ref{assumption}), we conclude that
\[
\y_{\ell}=\V_{21} \t_{\ell} + \V_{22}\s_{\ell} = (\V_{21}  + \V_{22}\G ) \t_{\ell} =\P\x_{\ell}
=\P (\V_{11} \t_{\ell} + \V_{12}\s_{\ell}) =\P (\V_{11} + \V_{12} \G ) \t_{\ell}
\]
holds for all $\t_{\ell}$, therefore
\[
(\V_{21}  + \V_{22}\G ) =\P (\V_{11} + \V_{12} \G ) 
\]
or
\begin{equation}
\P = (\V_{21}  + \V_{22}\G ) (\V_{11} + \V_{12} \G ) ^{-1}.
\end{equation}
Note that $\G \rightarrow 0$ as $N \rightarrow \infty$. This finishes the proof.
\hfill \qed
\end{proof}

Since the eigen-decomposition is not numerically stable. We
suggest using the Schur decomposition instead. Since $\boldsymbol{\Gamma}_k$ 
is symplectic, Corollary \ref{cor1} claims that there is an
orthogonal matrix $\W_k$ such that
\begin{equation}
\left[ \begin{array}{cc} \W_{11k} &  \W_{12k}  \\
\W_{21k}   &    \W_{22k}
\end{array} \right]^{\Tr}\boldsymbol{\Gamma}_k 
\left[ \begin{array}{cc} \W_{11k} &  \W_{12k}  \\
\W_{21k}   &    \W_{22k}
\end{array} \right]=
\left[ \begin{array}{cc}
\S_{11k} &  \S_{12k}  \\
\0   &    \S_{22k}
\end{array} \right],
\label{decomp1}
\end{equation}
where $\S_{11k}$ is upper-triangular and has all of its eigenvalues outside the unique circle.
We have the main result of the paper as follows.
\begin{theorem}
Let the Schur decomposition of $\boldsymbol{\Gamma}_k$ is given by (\ref{decomp1}).
The solution of the Riccati
equation corresponding to (\ref{inverse}) is given by
\begin{equation}
\P_k =\W_{21k}\W_{11k}^{-1}
\label{solution}
\end{equation}
\end{theorem}
\begin{proof}
The proof follows exactly the same argument of \cite[Remark 1]{laub79} and %therefore omitted.
it is provided here for completeness. From (\ref{jordan}), we have
\begin{equation}
\boldsymbol{\Gamma}_k
\left[ \begin{array}{c}
\V_{11k} \\ \V_{21k} 
\end{array} \right]
=\left[ \begin{array}{c}
\V_{11k} \\ \V_{21k} 
\end{array} \right]
\boldsymbol{\Delta}_k.
\label{comp1}
\end{equation}
From (\ref{decomp1}), we have
\[
\boldsymbol{\Gamma}_k
\left[ \begin{array}{c}
\W_{11k} \\ \W_{21k} 
\end{array} \right]
=\left[ \begin{array}{c}
\W_{11k} \\ \W_{21k} 
\end{array} \right]
\S_{11k}.
\]
Let $\T$ be an invertiable transformation matrix such that
\[
\T^{-1} \S_{11k} \T = \boldsymbol{\Delta}_k,
\]
then we have
\begin{equation}
\boldsymbol{\Gamma}_k \left[ \begin{array}{c}
\W_{11k} \\ \W_{21k} 
\end{array} \right] \T
= \left[ \begin{array}{c}
\W_{11k} \\ \W_{21k} 
\end{array} \right] \T \T^{-1}  \S_{11k} \T 
=\left[ \begin{array}{c}
\W_{11k} \\ \W_{21k} 
\end{array} \right] \T \boldsymbol{\Delta}_k
\label{comp2}
\end{equation}
Comparing (\ref{comp1}) and (\ref{comp2}) we must have 
\[
\left[ \begin{array}{c}
\W_{11k} \\ \W_{21k} 
\end{array} \right] \T 
= \left[ \begin{array}{c}
\V_{11k} \\ \V_{21k} 
\end{array} \right] \D 
\]
where $\D$ is a diagonal and invertiable matrix. Thus
\[
\W_{21k} \W_{11k}^{-1} = \V_{21k} \D \T^{-1} \T
\D^{-1} \V_{11k}^{-1} =\V_{21k} \V_{11k}^{-1}.
\]
This finishes the proof.
\hfill \qed
\end{proof}

Using either eigen-decomposition or Schur decomposition leads to the solution for (\ref{inverse}). 
But the formula of (\ref{solution}) is more stable than the formula of
(\ref{eigenSolu}) because Schur decomposition is a more stable 
process than eigen-decomposition.

We summarize the algorithm as follows.
\begin{algorithm} {\ } 
\begin{itemize}
\item[Step 0] Data: $\J$, $i_m$, $\Q$, $\R$, altitude of the 
spacecraft, and select sample period. 
\item[Step 1] Calculate $\A_k$ and $\B_k$ using (\ref{AB}-\ref{ABk}).
\item[Step 2] Calculate $\E_k$ and $\F_k$ using (\ref{Ek}-\ref{Fk}).
\item[Step 3] Calculate $\boldsymbol{\Gamma}_k$ using (\ref{Gamma}).
\item[Step 4] Use Schur decomposition (\ref{decomp1}) to get $\W_k$.
\item[Step 5] Calculate $\P_k$ using (\ref{solution}).
\end{itemize}
\label{mainAl}
\end{algorithm}

\section{Simulation test}

The proposed design algorithm has been tested for the following problem. 
Let the spacecraft inertia matrix be
\[
\J=\diag \left( 250, 150, 100 \right) kg \cdot m^2.
\]
The orbital inclination $i_m=57 C^o$, the orbit is circular with the altitude
$657$ km. In view of equation (\ref{period}), the orbital period is $5863$ seconds, and the orbital
rate is $\omega_0=0.0011$ rad/second. Assuming that the total number of samples
taken in one orbit is $100$, then, each sample period is $58.6352$ second.
Select $\Q=\diag(1.5*10^{-9},1.5*10^{-9},1.5*10^{-9},0.001,0.001,0.001)$
and $\R=\diag(2*10^{-3},2*10^{-3},2*10^{-3})$. We have calculated 
and stored $\P_k$ for $k=0,1,2,\ldots,99$ using Algorithm \ref{mainAl}.
Assuming that the initial quaternion error is $(0.01, 0.01, 0.01)$ and the
initial body rate is $(0.00001, 0.00001, 0.00001)$ radians per second,
applying the feedback (\ref{optiSolu}) to the system (\ref{discrete}), the simulated spacecraft
attitude response is given in Figures 1-6.

\begin{figure}[ht]
\centerline{\epsfig{file=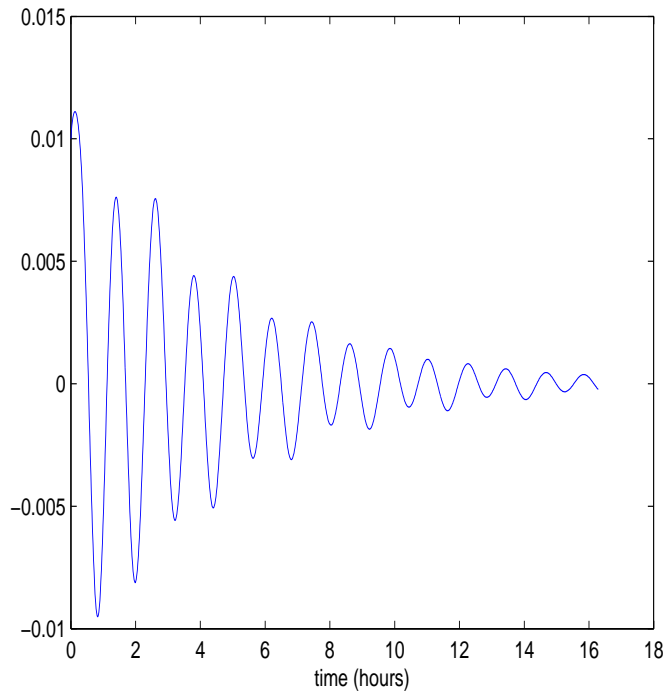,height=10cm,width=10cm}}
\caption{Attitude response $q_1$.}
\label{fig:q1}
\end{figure}

\begin{figure}[ht]
\centerline{\epsfig{file=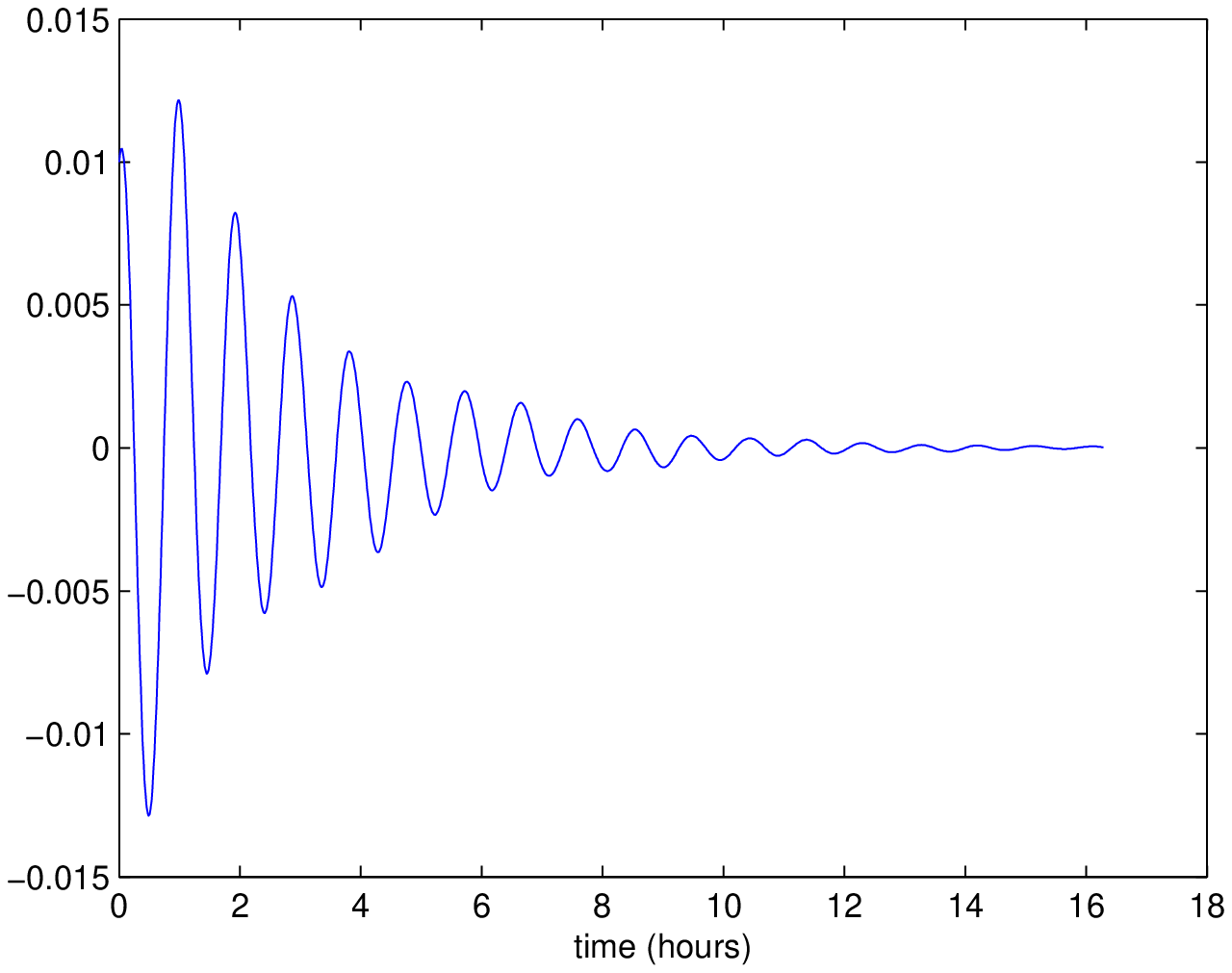,height=10cm,width=10cm}}
\caption{ Attitude response $q_2$.}
\label{fig:q2}
\end{figure}

\begin{figure}[ht]
\centerline{\epsfig{file=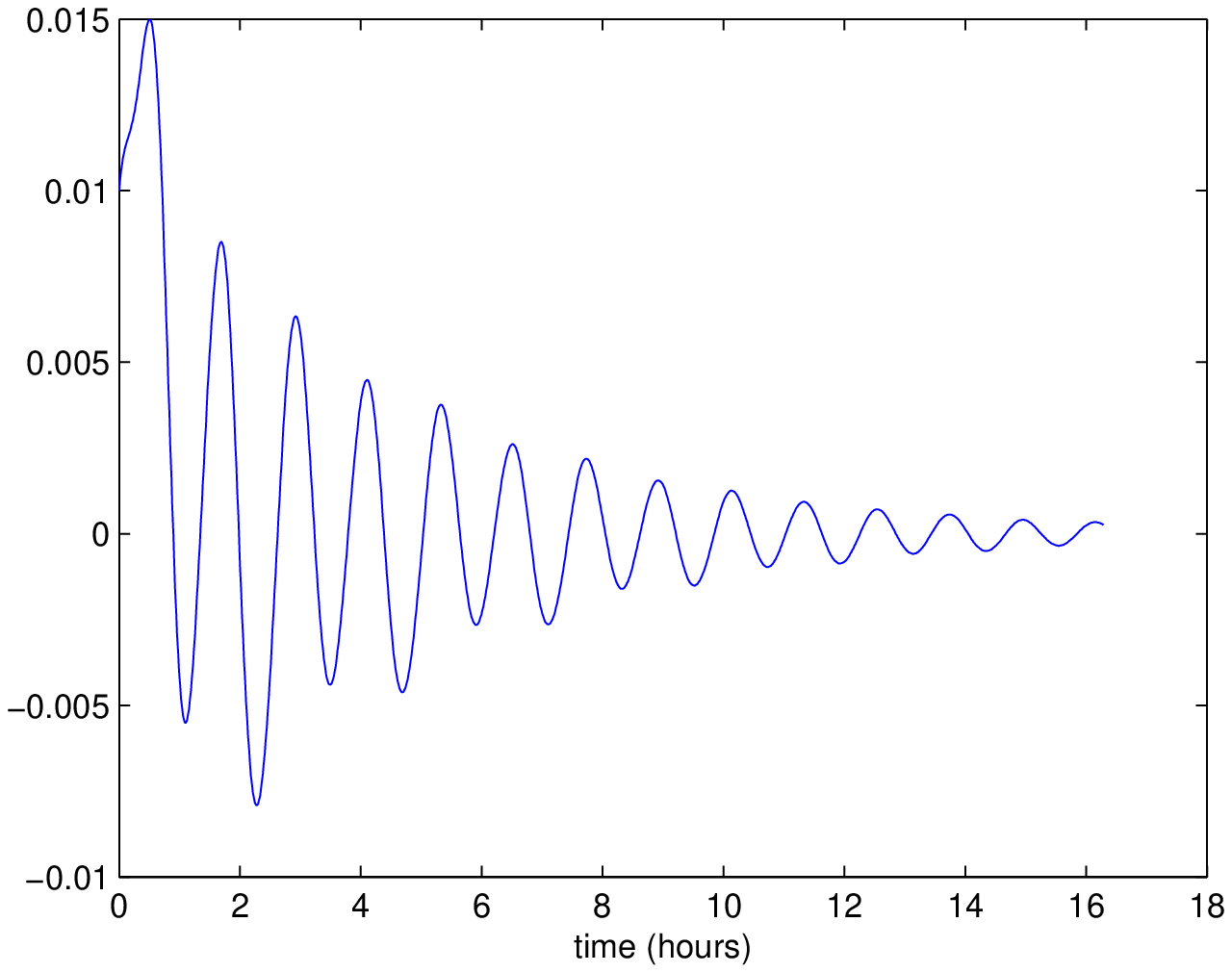,height=10cm,width=10cm}}
\caption{Attitude response $q_3$.}
\label{fig:q3}
\end{figure}

\begin{figure}[ht]
\centerline{\epsfig{file=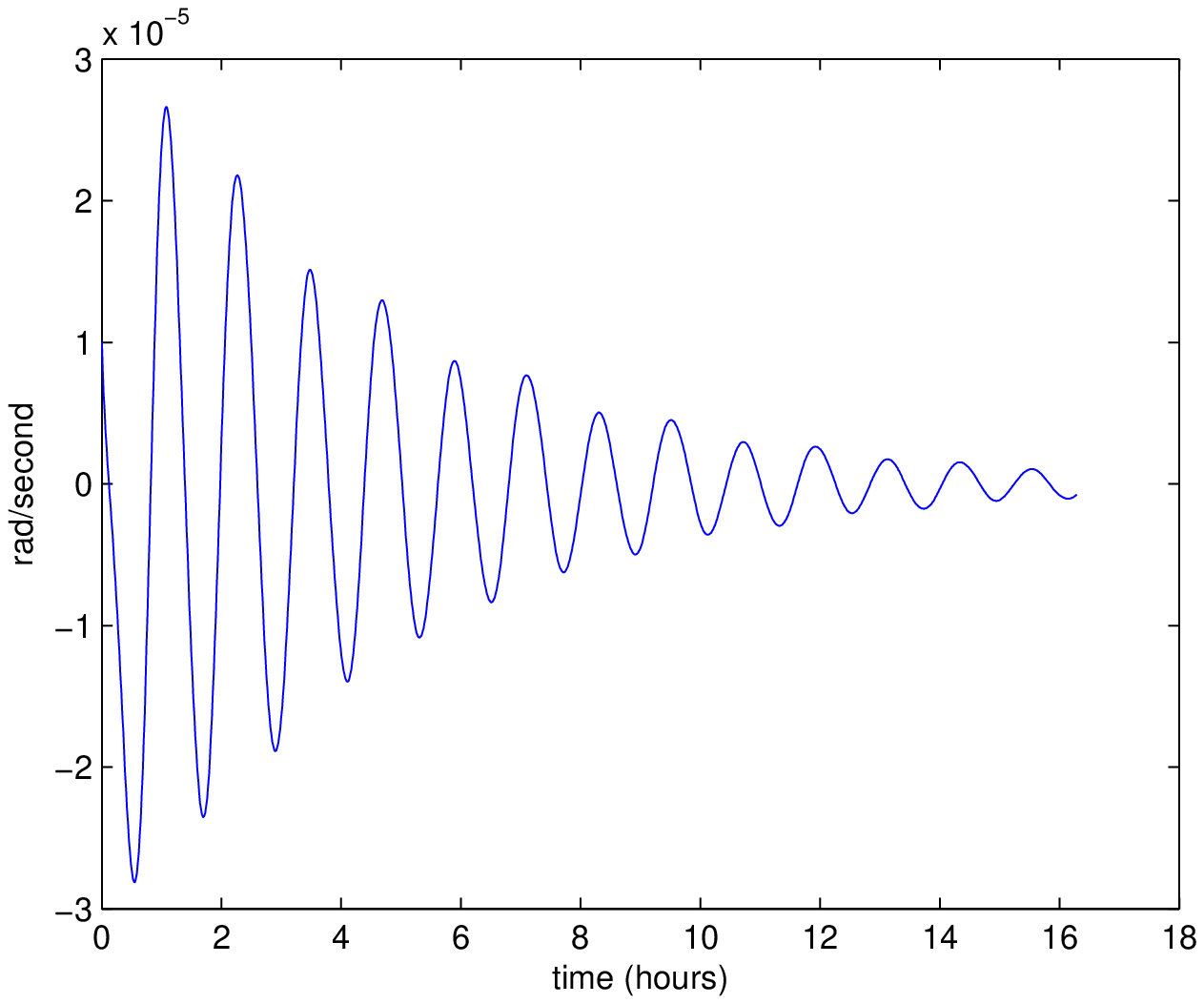,height=10cm,width=10cm}}
\caption{Body rate response $\omega_1$.}
\label{fig:w1}
\end{figure}

\begin{figure}[ht]
\centerline{\epsfig{file=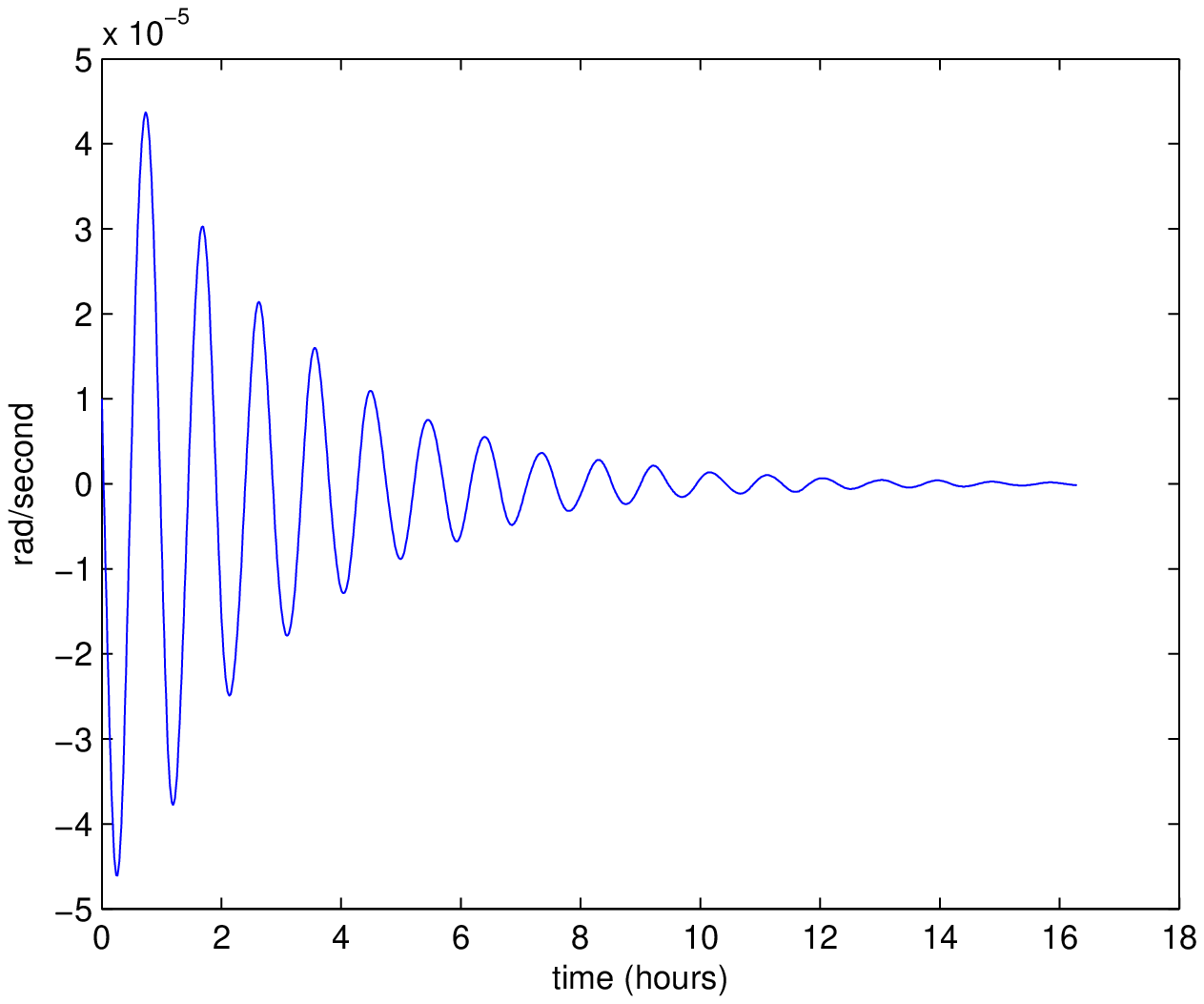,height=10cm,width=10cm}}
\caption{Body rate response $\omega_2$.}
\label{fig:w2}
\end{figure}

\begin{figure}[ht]
\centerline{\epsfig{file=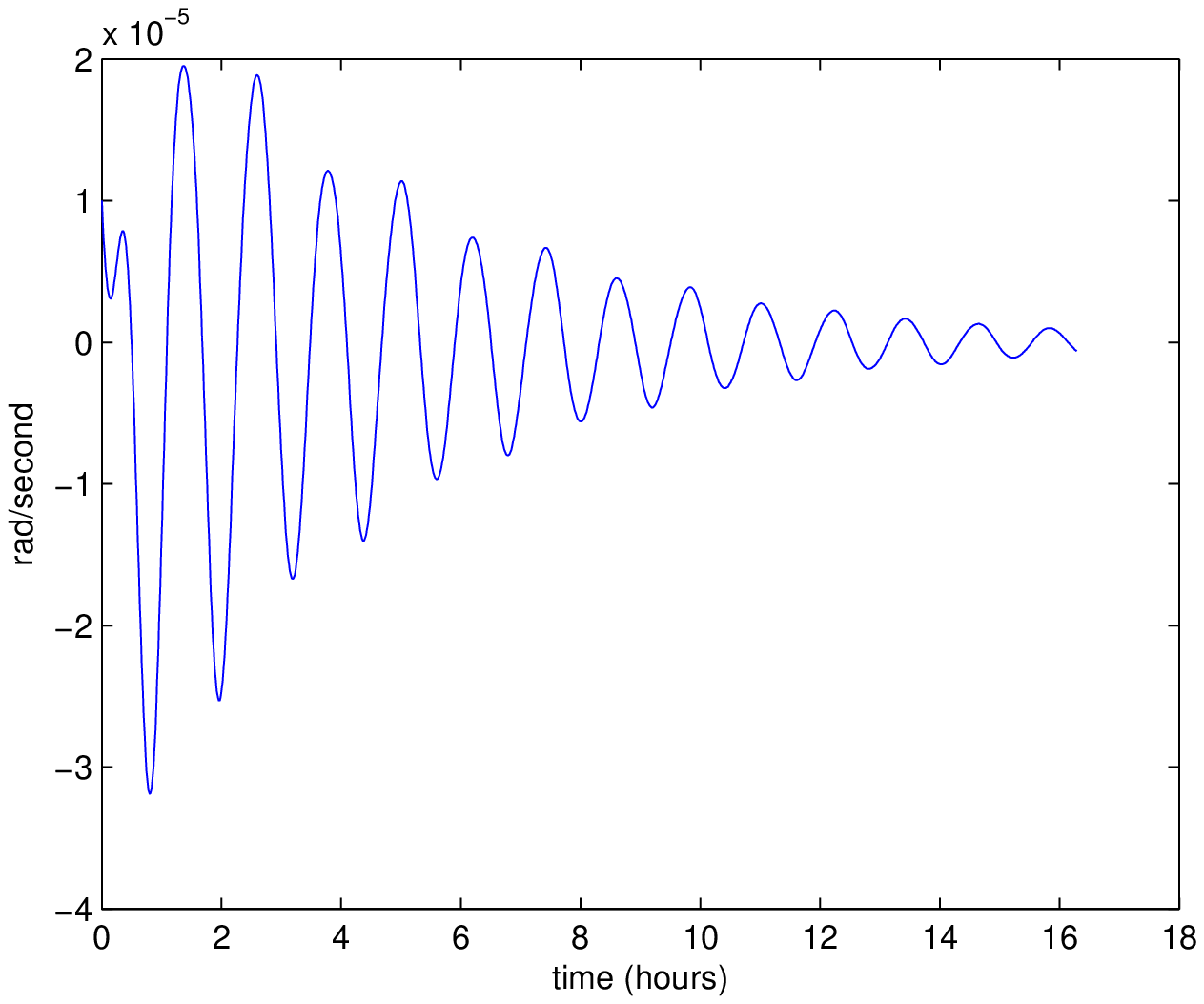,height=10cm,width=10cm}}
\caption{Body rate response $\omega_3$.}
\label{fig:w3}
\end{figure}

The designed controller stabilizes the spacecraft using only magnetic torques.
This shows the effectiveness of the design method. Since this time-varying 
system has a long period $5863$ seconds and the number of samples in 
each period is $100$, this means that using $\boldsymbol{\Gamma}_k$ in (\ref{Gamma})
instead of $\boldsymbol{\Pi}_k$ in (\ref{pi}) saves about $100$ matrix inverses, a significant
saving in the computation comparing to the well-known algorithm
\cite{hl94}! 

\section{Conclusion}
In this paper, we proposed a new algorithm for the 
design of the periodic controller for the spacecraft using only
magnetic torques. The proposed method is more efficient than
existing methods because it makes full use of the  information associated with this particular time-varying system. A simulation 
example is provided to demonstrate the effectiveness and fficiency 
of the algorithm. Although the algorithm is motivated by the 
problem of spacecraft attitude control using only magnetic
torques, it can be used in any time-verying system $(\A,\B(t))$
where only $\B(t)$ is a periodically time-varying matrix.

%\section*{Appendix}

%\section*{Acknowledgments}

\end{document}